\documentclass{article}
\usepackage[utf8]{inputenc}
\usepackage{amsthm,amsfonts,amssymb,amsmath,epsf, verbatim}

\usepackage{mathtools}

\usepackage[ruled,vlined]{algorithm2e}
\usepackage{tikz}
\usetikzlibrary{graphs}
\usetikzlibrary{graphs.standard}

\DeclarePairedDelimiter{\ceil}{\lceil}{\rceil}
\newtheorem{theorem}{Theorem}
\newtheorem{lemma}[theorem]{Lemma}

\newtheorem{proposition}[theorem]{Proposition}

\newtheorem{conjecture}[theorem]{Conjecture}
\newtheorem{question}[theorem]{Question}
\newcommand{\ch}[1]{\chi_{td}(#1)}
\newcommand{\clone}[1]{\mathrm{cl}(#1)}

\title{Total Difference Chromatic Numbers of Regular
Infinite Graphs}
\author{{\bf Noam Benson-Tilsen}\\ benson@rpi.edu \\ {\bf Samuel Brock}\\ s.evensong187@gmail.com \\ {\bf Brandon Faunce}\\ bfaunce28@gmail.com \\ {\bf Monish Kumar}\\ monishkumar833@gmail.com \\  {\bf Noah Dokko Stein}\\ noah.stein@yale.edu \\{\bf Joshua Zelinsky} \\jzelinsky@hopkins.edu \\  }
\date{}

\begin{document}

\maketitle

\begin{abstract}  Given a graph $G$,  a \textit{$k$-total difference labeling} of the graph is a total labeling $f$ from the set of edges and vertices to the set $\{1, 2, \cdots k\}$ satisfying that for any edge $\{u,v\}$, $f(\{u,v\})=|f(u)-f(v)|$. If $G$ is a graph, then $\chi_{td}(G)$ is the minimum $k$ such that there is a $k$-total difference labeling of $G$ in which no two adjacent labels are identical. We extend prior work on total difference labeling by improving the upper bound on $\chi_{td}(K_n)$ and also by proving results concerning infinite regular graphs. 
\end{abstract}

 By a $k$-{\it vertex labeling} of a graph, we mean a function $f$ from the vertices to the positive integers $\{1, 2, \cdots k\}$ for some $k$. Similarly, by a $k$-{\it edge labeling} of a graph, we mean a function  $f$ from the edges to  $\{1, 2, \cdots k\}$ for some $k$. A $k$-{\it total labeling} is a function $f$ from the set of edges and vertices to $\{1, 2, \cdots k\}$ for some $k$. A  $k$-vertex labeling is said to be {\it proper} if no two adjacent vertices share the same label. Similarly, a $k$-edge labeling is proper if no two edges that share a vertex share a label. A {\it proper $k$-total labeling} is a $k$-total labeling such that its corresponding $k$-edge labeling is proper, its corresponding $k$-vertex labeling is proper, and no edge has the same label as either of its vertices. \\
 
 Ranjan Rohatgi and Yufei Zhang introduced the idea of a total difference labeling of a graph \cite{RZ}.\\
 
 Given a graph $G$,  a \textit{$k$-total difference labeling} of the graph is a total labeling $f$ from the set of edges and vertices to the set $\{1, 2, \cdots k\}$ satisfying that for any edge $\{u,v\}$, $f(\{u,v\})=|f(u)-f(v)|$.  Recall that a total labeling of a graph is a labeling of both the edges and vertices of a graph. In general, $f$ is a function from the union of the edge set and vertex set of $G$ (denoted by $E(G)$ and $V(G)$, respectively) to the set $\{1,2,\cdots,k\}$. We will concern ourselves with proper total difference labelings. In a proper $k$-total difference labeling, $f$ is a function from $V(G) \cup E(G)$ to the set $\{1,2,\cdots,k\}$ that  satisfies the following properties:
 \begin{enumerate}
     \item For any edge $\{u,v\}$, $f(\{u,v\})=|f(u)-f(v)|$.
     \item No two adjacent vertices have the same same label. That is, if $\{u,v\}$ is an edge, then $f(u) \neq f(v)$.
     \item No two adjacent edges have the same label. That is, if $\{u,v\}$ and $\{v,w\}$ are edges, then $f(\{u,v\}) \neq f(\{v,w\})$.
     \item No vertex has the same label as an edge incident with it. That is, if $\{u,v\}$ is an edge, then $f(u) \neq f(\{u,v\})$.
 \end{enumerate}
 
 Property (1) is the defining property of total difference labelings, while properties (2), (3), and (4) make such a labeling proper. Note that the edge labels of a total difference labeling are determined by the vertex labels. Thus, we will often abuse notation and simply refer to the labeling of the vertices as the total difference labeling.  We will typically abbreviate ``total difference labeling" to ``TDL"; we will also typically omit ``proper" when referring to proper TDLs, and unless stated otherwise a TDL may be assumed to be proper.\\
 
 Rohatgi and Zhang define $\ch{G}$ as the smallest $k$ such that $G$ has a proper $k$-total difference labeling. They calculated $\ch{G}$ for a variety of graphs, including stars, wheels, and helms, as well as providing upper and lower bounds on $\chi_{td}$ of other graphs, such as trees and complete graphs. This paper extends their work in three main ways.\\
 
 First, we provide an essentially complete description of $\ch{K_n}$ for complete graphs and use this to bound $\ch{G}$ in general for any graph $G$ in terms of its order. To do this, we introduce the idea of a specialized set of numbers we call a {\it well-spaced row} (occasionally abbreviated to ``WSR").\\
 
 Second, we calculate $\ch{G}$ for various well-behaved infinite graphs, including the infinite square lattice. We also provide lower and upper bounds for $\ch{G}$ for some other infinite graphs.\\
 
 Third, we estimate $\ch{Q_n}$ for the hypercube graph $Q_n$.\\
 
 This paper is divided into five sections. In the first section, we review aspects of Rohatgi and  Zhang's work. We also discuss other similar labeling rules.\\
 
 In the second section, we introduce the ideas of well-spaced rows and star-elimination, which are major techniques that will be used throughout the rest of the paper.\\
 
 In the third section, we calculate $\chi_{td}$ for various infinite graphs, in particular the square lattice, the hexagonal lattice, the triangular lattice, and the infinite binary tree. We also give upper and lower bounds for the cubic lattice. \\
 
 
 In the fourth section, we introduce the idea of a clone of a graph, and use this to estimate $\ch{Q_n}$ of a hypercube. \\
 
 In the fifth section, we introduce the idea of saturated labelings and saturable graphs. This leads to two classes of graphs whose TDLs are particularly nice: saturable graphs and supersaturatable graphs. \\
 
 
 \section{Earlier work}
 
 Fundamentally, all of the different labeling schemes under discussion can be thought of as extensions of graph and edge labeling. Classically, given a graph $G$, $\chi(G)$, called the coloring number of $G$, denotes the minimum number of distinct colors needed to label every vertex of $G$ such that no two adjacent vertices are the same color. Instead of using colors one can use a \textit{labeling function} $f$ from the vertices to $\{1, 2 \cdots k\}$, which allows a much more natural framework. Thus, when we speak of a  ``vertex coloring'' of a graph we will mean a labeling of a graph's vertices with positive integers. One then defines the coloring number $\chi(G)$ as the minimum $k$ such that there is a labeling function $f$ with the property that $f(u) \neq f(v)$ when $u,v\in{V(G)}$ are adjacent. In this context, the most notable result is of course the famous Four Color Theorem, which says that any planar graph $G$ satisfies $\chi(G) \leq 4$. \\
 
 One major result about the behavior of $\chi(G)$ is Brooks's theorem \cite{Brooks}, which states that for any graph $G$, $\chi(G) \leq \Delta(G)+1$, with equality if and only if $G$ is a complete graph or a cycle. Here $\Delta(G)$ is the maximum degree of any vertex of $G$.  \\ 
 
 Similarly, $k$-edge labelings have been investigated. Define $\chi'(G)$ to be the minimum $k$ such that $G$ has a proper $k$-edge labeling. A classic result of Vizing \cite{Vizing} states that $\Delta(G) \leq \chi'(G) \leq \Delta(G)+1$. \\

 Let $\chi''(G)$ be the minimum $k$ such that a proper $k$-total labeling of $G$ exists.  One of the major open problems in this area is the {\it total coloring conjecture}, which states that $\chi''(G) \leq \Delta(G)+2$. \\
 
 In the last few years, a variety of papers, such as \cite{RZ}, have looked at various ways to combine edge colorings and vertex colorings where the edge colors are a function of the vertex colors. Another recent example is \cite{SLN}, which defined an {\it edge-coloring $k$-vertex weighting} as a function $f$ from the vertices and edges of $G$ to $\{1, 2 \cdots k\}$ where for any edge $\{u,v\}$, $f(\{u,v\}) = f(u)+f(v)$, and the corresponding edge labeling is proper. They then defined $\mu'(G)$ of a graph $G$ as the minimum $k$ such that an edge-coloring $k$-vertex weighting exists. One can similarly define $\mu''(G)$ which is the minimum $k$ such that there is an edge-coloring $k$-vertex weighing which is also a proper vertex labeling. It is not hard to show that $\ch{G} \geq \mu''(G) \geq \mu'(G)$, since any total difference labeling will also be a total labeling and an edge-coloring $k$-vertex weighting.  \\
 
 For the remainder of this paper, we will concern ourselves only with total difference labelings and their behavior; recall that we will often omit the word ``proper". \\
 
 One of the basic results of \cite{RZ} is a description of total difference labeling just in terms of the vertices by introducing {\emph{ doubles}} and {\emph{triples}}. Rohatgi and Zhang proved that a total difference labeling is proper if and only if it is a proper $k$-total labeling and it avoids doubles and triples. What do we mean by doubles and triples? \\
 
 Let $f$ be a not necessarily proper total difference labeling function of $G$.  A double is a pair of adjacent vertices $u$ and $v$ with $f(u)=2f(v)$ (see Figure \ref{double}). Note that if there exists a double $\{u,v\}$ in $G$, $f(\{u,v\}) = f(u)-f(v) = 2f(v)-f(v) = f(v)$; then $f(v)$ has the same label as $f(\{u,v\})$, and the labeling due to $f$ is not a total difference labeling.  \\
  \begin{figure}
     \centering
     \begin{tikzpicture}[node distance = {45mm}, thick, main/.style = {draw, circle}]
     \node[main] (1) {a};
     \node[main] (2) [right of=1] {2a};
     \draw (1) -- (2) node [midway, above] {$2a-a=a$};
     \end{tikzpicture}
     \caption{A double.}
     \label{double}
 \end{figure}
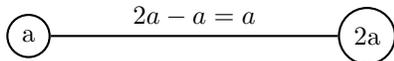
 
 There are two species of triple. The first is a set of three vertices $u$, $v$, and $w$ with $u$ adjacent to $v$ and $v$ adjacent to $w$ satisfying $f(u) = f(w)$. Notice that regardless of what label is given to $v$, we will then have $f(\{u,v\}) = f(\{v,w\})$. If we have this arrangement, we do not have a proper TDL. We will call this type of triple a {\emph{sandwich}} (see Figure \ref{sandwich}). \\
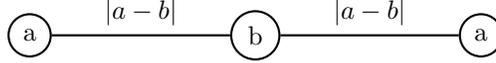
\begin{figure}
     \centering
     \begin{tikzpicture}[node distance = {30mm}, thick, main/.style = {draw, circle}]
     \node[main] (1) {a};
     \node[main] (2) [right of=1] {b};
     \node[main] (3) [right of=2] {a};
     \draw (1) -- (2) node [midway, above] {$|a-b|$};
     \draw (2) -- (3) node [midway, above] {$|a-b|$};
     \end{tikzpicture}
     \caption{A sandwich.}
     \label{sandwich}
 \end{figure}
 
 The second type of triple is a set of three vertices $u$, $v$, and $w$ with $u$ adjacent to $v$ and $v$ adjacent to $w$ such that $f(u)$, $f(v)$, and $f(w)$ form an arithmetic progression. This would also cause $f(\{u,v\}) = f(\{v,w\})$, in which case we would not have a proper TDL. We will call this type of triple a \textit{staircase} (see Figure \ref{staircase}).\\
 \begin{figure}
     \centering
     \begin{tikzpicture}[node distance = {45mm}, thick, main/.style = {draw, circle}]
     \node[main] (1) {a};
     \node[main] (2) [right of=1] {a+b};
     \node[main] (3) [right of=2] {a+2b};
     \draw (1) -- (2) node [midway, above] {$|a-(a+b)|=b$};
     \draw (2) -- (3) node [midway, above] {$|a+b-(a+2b)|=b$};
     \end{tikzpicture}
     \caption{A staircase.}
     \label{staircase}
 \end{figure}
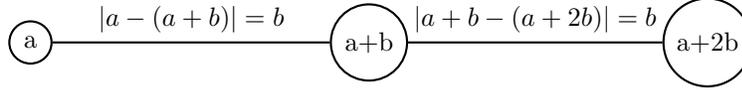
 
 Rohatgi and Zhang proved that a total difference labeling is proper if and only if the labeling has no doubles or triples. They did so by using the equivalence of the above description of doubles and triples with respect only to vertex labels and the definition of doubles and triples using both vertex labels and edge labels. \\

  Rohatgi and Zhang proved a variety of bounds on $\chi_{td}$ for different graphs. Here, we summarize the bounds we use or improve on this paper. 
  
 
 \begin{proposition}\label{RZ Xtd inequality for subgraphs} {\normalfont (Proposition 2.11, \cite{RZ})} Let $G'$ be a subgraph of $G$. Then $\ch{G'} \leq \ch{G}$.
 \end{proposition}
 
 \begin{proposition}\label{RZ upper bound for Kn}{\normalfont(Proposition 2.5,\cite{RZ})} Let $G$ be a graph with $n$ vertices. Then $\ch{G} \leq 3^{n-1}$.
 \end{proposition}

Due to Proposition \ref{RZ Xtd inequality for subgraphs}, Proposition \ref{RZ upper bound for Kn} is an upper bound on $\chi_{td}$ of the complete graph $K_n$, which is the graph with the highest $\chi_{td}$ of the graphs with $n$ vertices (as all such graphs are subgraphs of $K_n$). We will in the next section construct a substantially better bound on $\ch{K_n}$.

\begin{proposition}\label{RZ path result} {\normalfont(Theorem 3.1, \cite{RZ})} Let $n \geq 4$ and let $P_n$ be the path on $n$ vertices. Then $\ch{P_n} = 4$.
\end{proposition}

\begin{proposition}\label{RZ cycle result} \normalfont{(Theorem 3.2, \cite{RZ})} Let $n>2$ and let $C_n$ be the cycle on $n$. Then $\ch{C_n} = 4$ if $n \equiv 0$ (mod 3) and $\ch{C_n} = 5$ otherwise.
\end{proposition}

\begin{proposition}\label{RZ star result} {\normalfont(Theorem 4.1, \cite{RZ})} Let $K_{1,m}$ be the star graph with $m$ neighbors of the central vertex. Then $\ch{K_{1,m}} = m+1 $ when $m$ is even and $\ch{K_{1,m}} = m+2 $  when $m$ is odd. 

\end{proposition}
Although not mentioned in \cite{RZ}, it is worth noting that it is an immediate consequence of Proposition \ref{RZ cycle result} and Proposition \ref{RZ star result} together with Brooks's theorem that for any connected graph $G$ with more than one vertex, one has $\chi(G) < \ch{G}$.

They also proved a lower bound for a large class of graphs. 

\begin{proposition}\label{Diameter 2 bound}{\normalfont (Proposition 2.8, \cite{RZ})} Let $G$ be a graph with $n$ vertices where the diameter of $G$ is at most 2. Then $\ch{G} \geq n$, and all the vertex labels of $G$ must be distinct in any total difference labeling of $G$.
\end{proposition}
 
 \section{Well-spaced rows and star-elimination}
 We define a {\it well-spaced row} to be a set of positive integers such that no element is twice another element and no three elements form an arithmetic progression. Well-spaced rows are useful in finding upper bounds on $\ch{G}$, as their elements avoid doubles and staircases. \\
 
 A {\it minimal well-spaced row} is a well-spaced row that, given a finite cardinality, has the least possible maximum element. This is different from a {\it greedy well-spaced row}, which is generated from the {\it greedy well-spaced row algorithm}. The greedy well-spaced row algorithm takes in a number of elements $n$ and outputs a well-spaced row in the following manner. At each step, the algorithm appends the least positive integer such that the resulting set is a well-spaced row; it repeats until the set, a well-spaced row, has $n$ elements. For example, if we ask the algorithm for a 4-element greedy well-spaced row, it does the following:
 \begin{enumerate}
     \item Append 1.
     \item Check 2. Do not append it, because it would form a double with 1.
     \item Check 3. Append 3.
     \item Check 4. Append 4.
     \item Check 5. Do not append it, because it would form a staircase with 1 and 3.
     \item Check 6. Do not append it, because it would form a double with 3.
     \item Check 7. Do not append it, because it would form a staircase with 1 and 4.
     \item Check 8. Do not append it, because it would form a double with 4.
     \item Check 9. Append 9.
     \item Output the greedy well-spaced row with 4 elements: $\{1,3,4,9\}$.
 \end{enumerate}
 In pseudo-code, the greedy well-spaced row algorithm is the following:
 
 \begin{algorithm}
 \SetAlgoLined
 \SetKwInOut{Input}{input}\SetKwInOut{Output}{output}
 \Input{A number $n$ of elements in the output greedy WSR}
 \Output{A greedy WSR with $n$ elements}
 array $\leftarrow$ [\ ];\\
 \eIf{$n = 1$}{
 return array;
 }{
 temp $\leftarrow$ greedyWSR(n-1);\\
 max\_value $\leftarrow$ max(temp);\\
 n $\leftarrow$ max\_value + 1;\\
 appendToTemp(n);\\
 \nl\While{not isWSR(temp)}{
 removeFromTemp(n);\\
 $n = n + 1$;\\
 appendToTemp(n);\\}
 return temp;
 }
 \caption{Greedy well-spaced row algorithm}
\end{algorithm}

The above algorithm is equivalent to the following algorithm: For each positive integer $k$ (with $1\leq{k}\leq{n}$), convert $k-1$ to binary, then output the result as if it were ternary (base 3). For example, the 13th element in a greedy well-spaced row is output as follows. 
\begin{enumerate}
    \item Convert 13 to binary: $1101_2$.
    \item Read $1101_3$, which is $37_{10}$.
\end{enumerate}

Closely connected to the idea of a well-spaced row is that of a \textit{non-averaging set}. A non-averaging set is a set $S$ of non-negative integers that includes 0 such that no element of $S$ is the average of two other elements of $S$. Non-averaging sets are equivalent to well-spaced rows with the addition of the element 0. In other words, if a set $S$ is a non-averaging set, then $S-\{0\}$ is a well spaced row. Given a well-spaced row $R$, then $R \cup \{0\}$ is a non-averaging set. This is because an illegal double in a well-spaced row is equivalent to an illegal arithmetic progression with first element 0 in a non-averaging set. Non-averaging sets have been previously studied. See for example, work by Moser \cite{Moser}.  The greedy approach to producing non-averaging sets was also studied by Odlyzko and Stanley \cite{OS}.\\
 
Note that the greedy algorithm does not in general produce a minimal well-spaced row. For example, the greedy well-spaced row of length 4 is $\{1,3,4,9\}$, but $\{1,3,7,8\}$ is also a well-spaced row of length 4, with greatest element $8$, which is less than $9$. \\
 
An immediate consequence of the greedy well-spaced row construction is the following.
 
 \begin{theorem}\label{General greedy bound for a graph} For any graph $G$ with at most $n$ vertices,  $\ch{G} \leq 3^{\ceil{\log_2 n}}$.
 \end{theorem}
 \begin{proof} Assume that $G$ has at most $n$ vertices. Let $k$ be the smallest positive integer such that $2^k \geq n$. Then we may label the vertices of $G$ with a subset of the labels of the greedy well-spaced row with $2^k$ elements, whose largest element is precisely $3^k$.
 \end{proof}
 
 Notice that $$3^{\ceil{\log_2 n}} < 3^{1+\log_2 n} = 3 n^{\log_2 3}, $$ and so Theorem \ref{General greedy bound for a graph} gives a much tighter bound than Proposition \ref{RZ upper bound for Kn}, which gave a bound exponential in $n$.\\
 
 Given $n$, we will write $OS(n)$ to be the largest element in the greedy well-spaced row of size $n$. We will write $E(n)$ to be the largest element in  any minimal well-spaced row of size $n$. We will write $D(n)$ to be the number of distinct well-spaced rows with $n$ elements and largest element $E(n)$. We write $Mi_1(n)$ to be the minimum number which appears in any well-spaced row with $n$ elements and maximum element $E(n)$. We write $Mi_2(n)$ to be the maximum of the smallest element in each minimal well-spaced spaced row with $n$ elements and maximum element $E(n)$. Note that  $$Mi_1(n) \leq Mi_2(n) \leq E(n) \leq OS(n).$$ Also, if $D(n)=1$, then one must have $Mi_1(n) = Mi_2(n)$.
 
 We have the following data:
  \begin{center}
\begin{tabular}{ |c|c|c|c|c|c| } 
 \hline
 $n$ & $OS(n)$ & $E(n)$ & $D(n)$ & $Mi_1(n)$ & $Mi_2(n)$ \\ 
1 & 1 & 1 & 1 & 1 & 1 \\
2 & 3 & 3 & 2 & 1 & 2 \\
3 & 4 & 4 & 1 & 1 & 1 \\
4 & 9 & 8 & 4 & 1 & 2 \\
5 & 10 & 10 & 7 & 1 & 2 \\
6 & 12 & 12 & 6 & 1 & 2 \\
7 & 13 & 13 & 1 & 1 & 1 \\
8 & 27 & 19 & 2 & 1 & 2 \\
9 & 28 & 23 & 2 & 1 & 1 \\
10 & 30 & 25 & 2 & 1 & 2 \\
11 & 31 & 29 & 1 & 2 & 2 \\
12 & 36 & 31 & 2 & 1 & 1 \\
13 & 37 & 35 & 2 & 1 & 1 \\
14 & 39 & 39 & 20 & 1 & 2 \\
15 & 40 & 40 & 1 & 1 & 1 \\
16 & 81 & 50 & 14 & 1 & 3 \\
17 & 82 & 53 & 2 & 1 & 2 \\
18 & 84 & 57 & 2 & 1 & 4 \\
19 & 85 & 62 & 2 & 1 & 2 \\
20 & 90 & 70 & 4 & 1 & 2 \\
 \hline
\end{tabular}
\end{center}

In the Online Encyclopedia of Integer Sequences, $OS(n)$ is sequence A005836 and $E(n)$ is A005047. Currently, $D(n)$, $Mi_{1}(n)$, and $Mi_{2}(n)$ do not appear to be in the database. 

Based on the above, a few questions present themselves which as far as we are aware have not been asked previously about well-spaced rows (or asked equivalently about non-averaging sets):

\begin{enumerate}
    \item Are there infinitely many $n$ where the greedy algorithm gives a well-spaced row with the same maximum element as the minimal well-spaced row of the same size? That is, is $OS(n)=E(n)$ for infinitely many $n$?
    \item Can $Mi_1(n)$ grow arbitrarily large? Can $Mi_2(n)$ grow arbitrarily large? Further, will $Mi_1(n)$ or $Mi_2(n)$ take on every positive integer value?
    \item Can $D(n)$ grow arbitrarily large, and if so will it take on every positive integer value? 
    \item Is $D(n)=1$ infinitely often?
    \item Is $Mi_1(n)=1$ infinitely often? Is $Mi_2(n)=1$ infinitely often? \\
    \item Are there infinitely many $n$ where $D(n) \neq 1$ and also $Mi_1(n)=Mi_2(n)$?

Since $E(n)$ is an increasing function, it is also natural to ask about the behavior of the first difference of $E(n)$. In particular, set $J(n) = E(n+1)-E(n)$. Then we may also ask similar questions about $J(n)$. Two questions then seem particularly natural:

\item Does $J(n)$ take on every positive integer value? 
\item Is $J(n)=1$ infinitely often? 

\end{enumerate}
 
 \textit{Star-elimination} is a method that can be used to find lower bounds on the total difference labeling numbers of regular infinite graphs. Recall that each vertex in a regular graph has the same degree. In an infinite regular graph $G$, each vertex is the center of a star subgraph $K_{1,\Delta}$ of $G$, where $\Delta$ is the degree of each vertex in the graph. Because $K_{1,\Delta}$ is a subgraph of an infinite $\Delta$-regular graph $\Omega$, Proposition \ref{RZ Xtd inequality for subgraphs} yields the quick lower bound $\ch{\Omega} \geq \ch{K_{1,\Delta}}$, which is either $\Delta + 1$ or $\Delta + 2$, as per Proposition \ref{RZ star result}. This lower bound is a starting point for the main part of the star-elimination method. \\

 Given a $\Delta$-regular infinite graph, assume $x = \Delta = \ch{G}$. We then try to show that it is impossible to label the relevant star subgraph $S=K_{1,\Delta}$ given this value and given the restriction that all labels must be distinct, as per Proposition \ref{Diameter 2 bound}. If a contradiction arises, we set $x+1=\ch{G}$ and repeat until this procedure fails, at which point the star-elimination method has produced a lower bound on the true value of $\ch{G}$. In most of our cases, the star-elimination method will produce the best-possible lower bound. \\
 
 Let $v$ be a vertex in $G$. Since $G$ is $\Delta$-regular, there must be exactly $\Delta$ vertices connected to $v$. This vertex will have a label, $l$, in the set of $\{1,2,3,...x\}$.\\

 We will now assign a label to each vertex adjacent to $v$. We cannot use the label $l$ for any of them, nor can we use $2l$ or $\frac{l}{2}$ since the vertex and $v$ would form a double, resulting in an improper labeling. We also cannot use both $l-a$ and $l+a$ for two vertices and any $a$, since $l-a$, $l$, and $l+a$ would form a staircase. Finally, we can't use the same label twice in any of the adjacent vertices since they would form a sandwich around $v$. Write down all the positive integers less than $l$ in a list, which we will call $L_1$, in descending order. Write down all the integers greater than $l$ and less than or equal to $x$ in a similar list, called $L_2$. If $L_{ji}$ is index i of list $L_j$, we know we can use either $L_{1i}$ or $L_{2i}$, but not both. In addition, we can't use $L_{ji}$ if it is either $\frac{l}{2}$, $l$, $2l$, or if it is less than 1 or greater than $x$. For each value of $i$, if either $L_{1i}$ or $L_{2i}$ exists, we can arbitrarily label one of the vertices adjacent to $v$ with one of them. This is the maximum number of vertices we can label for this $l$. If we are not able to label all of the adjacent vertices with this method for any values of $l$, then $\ch{G} > x$, since you must need more labels to complete a proper labeling. If we are able, $x$ is a definite lower bound for $\ch{G}$, since all smaller $x$ cannot be $\ch{G}$. 
 
 To find a contradiction in $x=\ch{G}$, we assume each value $k$ from 1 to $x$ is the label of the center of a star $S$, because in a regular infinite graph each label is assigned to the center of some such $S$.\\ 
 
 In each case, we find the number of vertex labels that could neighbor the center of $S$. We start with a given $k$ which is at most $x$, and assume that $k$ is the center of a star. We then list all  the possible labels, from 1 to $x$. We first note that $k$ itself is removed immediately from this list, as the center's label is $k$, so none of its neighbors can be labeled $k$.\\
 
 When $k$ is even we may also remove $\frac{k}{2}$ from our list, as it would create a double with $k$.  Similarly, we may  remove $2k$. Note that we may have only one element from each pair of numbers that produces a triple with $k$ as the second element. For example, if we have assumed 4 to be the center label, we must eliminate one of each of 1 and 7, 2 and 6, and 3 and 5 \textbf{---} note that 2 would have already been removed in the previous step. If the number of available labels remaining is less than the number of vertices $\Delta$ neighboring the center of $S$, then $k$ cannot be a vertex label. We repeat this procedure for all labels from 1 to $x$, marking each label as either ``possible" or ``impossible". The order in which we check the labels from 1 to $x$ is tactical, as eliminating some values of $k$ will depend on the prior elimination of other values of $k$. If at any point the total number of ``possible" labels for vertices is less than $\Delta+1$, then we have produced a contradiction and may conclude that $\ch{G} > x$.
 
 
 A more formal way of expressing star-elimination is as follows. Given a finite set of positive integers $A$, and $j$ an element of $A$, we will say that a subset $B$ of $A$ is $j$-acceptable if $B$ does not contain $2j$, $j/2$ or a three term arithmetic sequence with $j$ as the middle term.  We say that that $j$ is $k$-star-vulnerable from $A$ if the cardinality of any $j$-acceptable subset of $A$ is strictly less than $k$.
 
 A star-elimination sequence with respect to $k$ and $n$ then is a sequence of sets, $S_0$, $S_1$, $S_2, \cdots ,S_m$ satisfying three properties:
 \begin{enumerate}
 \item $S_0= \{1, 2, 3, ..., n\}$.
 \item For any $0 < i \leq m$, $S_i \subset S_{i-1}$.
 \item If $j \in S_{i-1}/S_i$ then $j$ is $k$-star vulnerable.\\
 \end{enumerate}
 The star-elimination method is essentially the observation that if there is a star-elimination sequence with respect to $k$ and $n$, then there cannot be a $k$-regular total difference labeling with all labels less than or equal to $n$.

 \section{Total difference labeling of infinite graphs}
 
 We begin by finding the total difference chromatic number of the infinite square lattice, which we denote by $\Omega_S$.

 \begin{theorem}  \label{Square lattice TDL is 8} We have $\ch{\Omega_S} = 8$.
 \end{theorem}
 \begin{proof} First, we use star-elimination to prove that $\ch{\Omega_S} \geq 8$. 
  Then we use well-spaced rows to show that $\ch{\Omega_S} \leq 8$. 
 
 Let us first find a lower bound using star-elimination. Because, by Proposition \ref{RZ star result}, $\ch{K_{1,4}}=4+1=5$ (note that each vertex is the center of a $K_{1,4}$), we have a lower bound of 5, so $\ch{\Omega_S}\geq5$. Assume $\ch{\Omega_S}=5$. Now using star-elimination it is straightforward to show that $\ch{\Omega_S}\neq5$ and $\ch{\Omega_S}\neq6$, so $\ch{\Omega_S}\geq7$.

 We now demonstrate star-elimination to show that $\ch{\Omega_S}>7$. Begin by assuming $\ch{\Omega_S}=7$. In this case, we will need to eliminate three of the seven potential labels $\{1,2,3,4,5,6,7\}$ to show a contradiction.
 \begin{enumerate}
     \item Assume that 4 is included somewhere in a 7-TDL of $\Omega_S$. The vertex $v$ with label 4 has four neighbors, each necessarily distinct (to avoid sandwiches). We first remove 4 from the list of possible labels, as $v$ cannot be adjacent to a vertex labeled 4. We also remove $\frac{4}{2}=2$, leaving $\{1,3,5,6,7\}$. We finally remove one of 1 and 7 as well as one of 3 and 5. This leaves only three possible labels for the four neighbors of $v$, which is a contradiction, so no vertex can be labeled with 4.
     \item Assume that 3 is the label of some vertex $v$ in $\Omega_S$. We remove 3, 4, and $3\times2=6$ from the list of possible labels of the neighbors of $v$, leaving $\{1,2,5,7\}$; we then remove one of 1 and 5, again leaving only three labels for the four neighbors of $v$, which is a contradiction, so no vertex can be labeled with 3. 
     \item Assume that 6 is the label of some vertex $v$ in $\Omega_S$. We remove 3, 4, and 6 from the list of possible labels of the neighbors of $v$, leaving $\{1,2,5,7\}$. We also remove one of 5 and 7, leaving three labels for the neighbors of $v$, which is a contradiction, so no vertex can be labeled with 6.
 \end{enumerate}
 Because this leaves only four distinct labels for a TDL of $\Omega_S$, we have a contradiction and $\ch{\Omega_S}>7$. Thus, $\ch{\Omega_S} \geq 8$. \\
 
 We now find an upper bound using well-spaced rows.  Clearly each star subgraph $K_{1,4}$ of $\Omega_S$ uses five distinct labels. The greatest element of a minimal well-spaced row with five elements is 10, as in $W=\{1,3,4,9,10\}$ (which incidentally is the same as the greedy well-spaced row of five elements). Labeling a ``row" of vertices in $\Omega_S$ (hence the name ``well-spaced rows") with the elements of $W$, then labeling an adjacent row in the same way but with labels shifted two vertices in either direction, and so on, gives a 10-TDL of $\Omega_S$, showing $\ch{\Omega_S}\leq10$ (see Figure \ref{square_lattice_10tdl}).\\
 
  \begin{figure}
     \centering
     \begin{tikzpicture}[node distance = {15mm}, thick, main/.style = {draw, circle}]
     \node[main] (1) {1};
     \node[main] (2) [right of=1] {3};
     \node[main] (3) [right of=2] {4};
     \node[main] (4) [right of=3] {9};
     \node[main] (5) [right of=4] {10};
     \node[main] (6) [below of=1] {9};
     \node[main] (7) [below of=2] {10};
     \node[main] (8) [below of=3] {1};
     \node[main] (9) [below of=4] {3};
     \node[main] (10) [below of=5] {4};
     \node[main] (11) [below of=6] {3};
     \node[main] (12) [below of=7] {4};
     \node[main] (13) [below of=8] {9};
     \node[main] (14) [below of=9] {10};
     \node[main] (15) [below of=10] {1};
     \draw (1) -- (2);
     \draw (1) -- (6);
     \draw (2) -- (3);
     \draw (2) -- (7);
     \draw (3) -- (4);
     \draw (3) -- (8);
     \draw (4) -- (5);
     \draw (4) -- (9);
     \draw (5) -- (10);
     \draw (6) -- (7);
     \draw (6) -- (11);
     \draw (7) -- (8);
     \draw (7) -- (12);
     \draw (8) -- (9);
     \draw (8) -- (13);
     \draw (9) -- (10);
     \draw (9) -- (14);
     \draw (10) -- (15);
     \draw (11) -- (12);
     \draw (12) -- (13);
     \draw (13) -- (14);
     \draw (14) -- (15);
     \end{tikzpicture}
     \caption{Subgraph of $\Omega_S$ and corresponding 10-total difference labeling.}
     \label{square_lattice_10tdl}
 \end{figure}
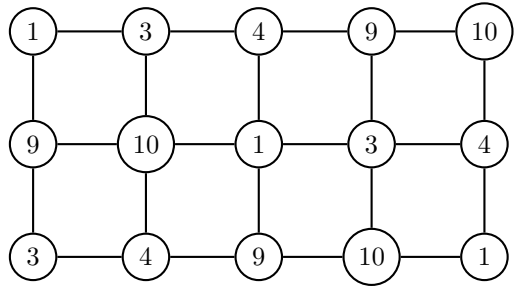

  We can improve our weak upper bound from 10 by increasing  the number of distinct labels being used. The six labels $\{1,2,3,6,7,8\}$  contain two potential doubles and two potential staircases. However, they can be arranged to form a valid 8-total difference labeling of $\Omega_S$, as shown in Figure \ref{square_lattice_8tdl}.
 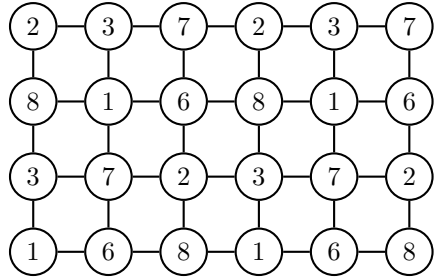
\begin{figure}
     \centering
     \begin{tikzpicture}[node distance = {10mm}, thick, main/.style = {draw, circle}]
     \node[main] (1) {2};
     \node[main] (2) [right of=1] {3};
     \node[main] (3) [right of=2] {7};
     \node[main] (4) [right of=3] {2};
     \node[main] (5) [right of=4] {3};
     \node[main] (6) [right of=5] {7};
     \node[main] (7) [below of=1] {8};
     \node[main] (8) [below of=2] {1};
     \node[main] (9) [below of=3] {6};
     \node[main] (10) [below of=4] {8};
     \node[main] (11) [below of=5] {1};
     \node[main] (12) [below of=6] {6};
     \node[main] (13) [below of=7] {3};
     \node[main] (14) [below of=8] {7};
     \node[main] (15) [below of=9] {2};
     \node[main] (16) [below of=10] {3};
     \node[main] (17) [below of=11] {7};
     \node[main] (18) [below of=12] {2};
     \node[main] (19) [below of=13] {1};
     \node[main] (20) [below of=14] {6};
     \node[main] (21) [below of=15] {8};
     \node[main] (22) [below of=16] {1};
     \node[main] (23) [below of=17] {6};
     \node[main] (24) [below of=18] {8};
     \draw (1) -- (2);
    \draw (2) -- (3);
    \draw (3) -- (4);
    \draw (4) -- (5);
    \draw (5) -- (6);
    \draw (7) -- (8);
    \draw (8) -- (9);
    \draw (9) -- (10);
    \draw (10) -- (11);
    \draw (11) -- (12);
    \draw (13) -- (14);
    \draw (14) -- (15);
    \draw (15) -- (16);
    \draw (16) -- (17);
    \draw (17) -- (18);
    \draw (19) -- (20);
    \draw (20) -- (21);
    \draw (21) -- (22);
    \draw (22) -- (23);
    \draw (23) -- (24);
    \draw (1) -- (7);
    \draw (2) -- (8);
    \draw (3) -- (9);
    \draw (4) -- (10);
    \draw (5) -- (11);
    \draw (6) -- (12);
    \draw (7) -- (13);
    \draw (8) -- (14);
    \draw (9) -- (15);
    \draw (10) -- (16);
    \draw (11) -- (17);
    \draw (12) -- (18);
    \draw (13) -- (19);
    \draw (14) -- (20);
    \draw (15) -- (21);
    \draw (16) -- (22);
    \draw (17) -- (23);
    \draw (18) -- (24);
     \end{tikzpicture}
     \caption{Subgraph of $\Omega_S$ and corresponding 8-total difference labeling.}
     \label{square_lattice_8tdl}
 \end{figure}
 
 \end{proof}

 Finding the total difference chromatic numbers of the infinite hexagonal and triangular lattices is similar. We denote the infinite hexagonal lattice by $\Omega_H$ and follow a procedure similar to the one used for $\Omega_S$, using well-spaced rows and star-elimination for lower and upper bounds.

\begin{theorem} We have $\ch{\Omega_H} = 7$.

\end{theorem}
\begin{proof}
We first find a lower bound on $\ch{\Omega_H}$ using star-elimination. Because $\Omega_H$ is 3-regular, each vertex is the center of a $K_{1,3}$; we therefore have a lower bound of $\ch{K_{1,3}}=3+2=5$. So we first assume $\ch{\Omega_H}=5$.
 It is again straightforward to show a contradiction in this case, so we will start by contradicting that $\ch{\Omega_H}=6$.
\begin{enumerate}
    \item Assume 3 appears in a 6-TDL of $\Omega_H$. We eliminate 3 and 6, as well as one of 1 and 5 and one of 2 and 4, leaving two possible labels for its neighbor, which is a contradiction; 3 does not appear.
    \item Assume 2 appears. We remove 1, 2, 3, and 4, leaving only 5 and 6, which is a contradiction; 2 does not appear.
    \item Assume 5 appears. We remove 2, 3, and 5, as well as one of 4 and 6, leaving two possible neighboring labels, which is a contradiction; 5 does not appear.
\end{enumerate}
We are left with the three labels $\{1,4,6\}$ for a 6-TDL of $\Omega_H$, which is impossible. Therefore $\ch{\Omega_H}>6$. Star-elimination fails to produce a contradiction for $\ch{\Omega_H}=7$, so $\ch{\Omega_H} \geq 7$.

Recall that each vertex of $\Omega_H$ is the center of a $K_{1,3}$. The greatest element of a minimal well-spaced row with four elements is 8, as in $W=\{1,3,7,8\}$. The construction of ``rows" is not as obvious for $\Omega_H$ as it is for $\Omega_S$, but we can structure $\Omega_H$ as a subgraph of $\Omega_S$ as follows: Construct $\Omega_S$, then remove alternating vertical edges within a row, then shift horizontally by one vertex, repeat in the adjacent rows, and so on, as shown below in Figure \ref{hex_subgraph_of_square}. 
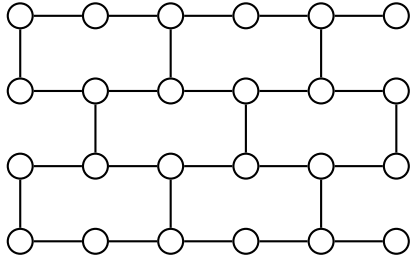
\begin{figure}
    \centering
    \begin{tikzpicture}[node distance = {10mm}, thick, main/.style = {draw, circle}]
    \node[main] (1) {};
     \node[main] (2) [right of=1] {};
     \node[main] (3) [right of=2] {};
     \node[main] (4) [right of=3] {};
     \node[main] (5) [right of=4] {};
     \node[main] (6) [right of=5] {};
     \node[main] (7) [below of=1] {};
     \node[main] (8) [below of=2] {};
     \node[main] (9) [below of=3] {};
     \node[main] (10) [below of=4] {};
     \node[main] (11) [below of=5] {};
     \node[main] (12) [below of=6] {};
     \node[main] (13) [below of=7] {};
     \node[main] (14) [below of=8] {};
     \node[main] (15) [below of=9] {};
     \node[main] (16) [below of=10] {};
     \node[main] (17) [below of=11] {};
     \node[main] (18) [below of=12] {};
     \node[main] (19) [below of=13] {};
     \node[main] (20) [below of=14] {};
     \node[main] (21) [below of=15] {};
     \node[main] (22) [below of=16] {};
     \node[main] (23) [below of=17] {};
     \node[main] (24) [below of=18] {};
     \draw (1) -- (2);
    \draw (2) -- (3);
    \draw (3) -- (4);
    \draw (4) -- (5);
    \draw (5) -- (6);
    \draw (7) -- (8);
    \draw (8) -- (9);
    \draw (9) -- (10);
    \draw (10) -- (11);
    \draw (11) -- (12);
    \draw (13) -- (14);
    \draw (14) -- (15);
    \draw (15) -- (16);
    \draw (16) -- (17);
    \draw (17) -- (18);
    \draw (19) -- (20);
    \draw (20) -- (21);
    \draw (21) -- (22);
    \draw (22) -- (23);
    \draw (23) -- (24);
    \draw (1) -- (7);
    \draw (3) -- (9);
    \draw (5) -- (11);
    \draw (8) -- (14);
    \draw (10) -- (16);
    \draw (12) -- (18);
    \draw (13) -- (19);
    \draw (15) -- (21);
    \draw (17) -- (23);
     \end{tikzpicture}
    \caption{Representation of $\Omega_H$ as a subgraph of $\Omega_S$.}
    \label{hex_subgraph_of_square}
\end{figure}
Assign a minimal well-spaced row with four elements, such as $W$, to each row of the graph, shifted horizontally by two vertices in each adjacent row, as in the upper-bound construction for $\ch{\Omega_S}$ in Figure \ref{square_lattice_10tdl}; see Figure \ref{hex_lattice_8tdl} for a subgraph of $\Omega_H$ with this labeling. Notice that this results in pairs of labels in each ``column" of vertices, which is a useful pattern.
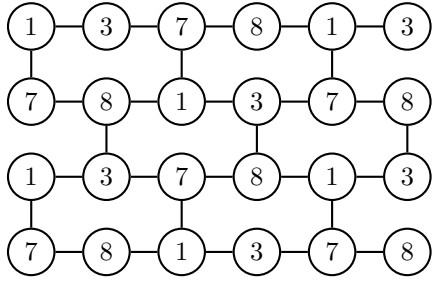
\begin{figure}
    \centering
    \begin{tikzpicture}[node distance = {10mm}, thick, main/.style = {draw, circle}]
    \node[main] (1) {1};
     \node[main] (2) [right of=1] {3};
     \node[main] (3) [right of=2] {7};
     \node[main] (4) [right of=3] {8};
     \node[main] (5) [right of=4] {1};
     \node[main] (6) [right of=5] {3};
     \node[main] (7) [below of=1] {7};
     \node[main] (8) [below of=2] {8};
     \node[main] (9) [below of=3] {1};
     \node[main] (10) [below of=4] {3};
     \node[main] (11) [below of=5] {7};
     \node[main] (12) [below of=6] {8};
     \node[main] (13) [below of=7] {1};
     \node[main] (14) [below of=8] {3};
     \node[main] (15) [below of=9] {7};
     \node[main] (16) [below of=10] {8};
     \node[main] (17) [below of=11] {1};
     \node[main] (18) [below of=12] {3};
     \node[main] (19) [below of=13] {7};
     \node[main] (20) [below of=14] {8};
     \node[main] (21) [below of=15] {1};
     \node[main] (22) [below of=16] {3};
     \node[main] (23) [below of=17] {7};
     \node[main] (24) [below of=18] {8};
     \draw (1) -- (2);
    \draw (2) -- (3);
    \draw (3) -- (4);
    \draw (4) -- (5);
    \draw (5) -- (6);
    \draw (7) -- (8);
    \draw (8) -- (9);
    \draw (9) -- (10);
    \draw (10) -- (11);
    \draw (11) -- (12);
    \draw (13) -- (14);
    \draw (14) -- (15);
    \draw (15) -- (16);
    \draw (16) -- (17);
    \draw (17) -- (18);
    \draw (19) -- (20);
    \draw (20) -- (21);
    \draw (21) -- (22);
    \draw (22) -- (23);
    \draw (23) -- (24);
    \draw (1) -- (7);
    \draw (3) -- (9);
    \draw (5) -- (11);
    \draw (8) -- (14);
    \draw (10) -- (16);
    \draw (12) -- (18);
    \draw (13) -- (19);
    \draw (15) -- (21);
    \draw (17) -- (23);
     \end{tikzpicture}
    \caption{Upper-bound 8-TDL of subgraph $\Omega_H$.}
    \label{hex_lattice_8tdl}
\end{figure}

 We now know $7 \leq \ch{\Omega_H} \leq 8$, and in fact we can find a construction with $\ch{\Omega_H}=7$, using all 7 labels $\{1,2,3,4,5,6,7\}$. We do this by labeling the columns of vertices with the pairs of labels $\{4,3\}$, $\{6,1\}$, $\{2,5\}$, $\{7,3\}$, $\{5,2\}$, and $\{1,7\}$, as shown in Figure \ref{hex_lattice_7tdl}, similar to the construction in Figure \ref{hex_lattice_8tdl}.
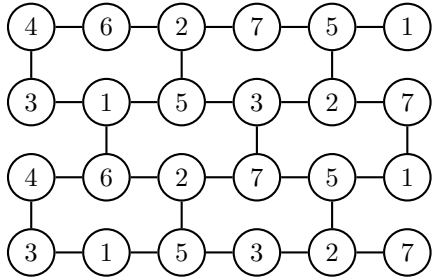
\begin{figure}
    \centering
    \begin{tikzpicture}[node distance = {10mm}, thick, main/.style = {draw, circle}]
    \node[main] (1) {4};
     \node[main] (2) [right of=1] {6};
     \node[main] (3) [right of=2] {2};
     \node[main] (4) [right of=3] {7};
     \node[main] (5) [right of=4] {5};
     \node[main] (6) [right of=5] {1};
     \node[main] (7) [below of=1] {3};
     \node[main] (8) [below of=2] {1};
     \node[main] (9) [below of=3] {5};
     \node[main] (10) [below of=4] {3};
     \node[main] (11) [below of=5] {2};
     \node[main] (12) [below of=6] {7};
     \node[main] (13) [below of=7] {4};
     \node[main] (14) [below of=8] {6};
     \node[main] (15) [below of=9] {2};
     \node[main] (16) [below of=10] {7};
     \node[main] (17) [below of=11] {5};
     \node[main] (18) [below of=12] {1};
     \node[main] (19) [below of=13] {3};
     \node[main] (20) [below of=14] {1};
     \node[main] (21) [below of=15] {5};
     \node[main] (22) [below of=16] {3};
     \node[main] (23) [below of=17] {2};
     \node[main] (24) [below of=18] {7};
     \draw (1) -- (2);
    \draw (2) -- (3);
    \draw (3) -- (4);
    \draw (4) -- (5);
    \draw (5) -- (6);
    \draw (7) -- (8);
    \draw (8) -- (9);
    \draw (9) -- (10);
    \draw (10) -- (11);
    \draw (11) -- (12);
    \draw (13) -- (14);
    \draw (14) -- (15);
    \draw (15) -- (16);
    \draw (16) -- (17);
    \draw (17) -- (18);
    \draw (19) -- (20);
    \draw (20) -- (21);
    \draw (21) -- (22);
    \draw (22) -- (23);
    \draw (23) -- (24);
    \draw (1) -- (7);
    \draw (3) -- (9);
    \draw (5) -- (11);
    \draw (8) -- (14);
    \draw (10) -- (16);
    \draw (12) -- (18);
    \draw (13) -- (19);
    \draw (15) -- (21);
    \draw (17) -- (23);
     \end{tikzpicture}
    \caption{Subgraph of $\Omega_H$ with 7-TDL. The pattern is apparent in columns of vertices.}
    \label{hex_lattice_7tdl}
\end{figure}

\end{proof}

We now find $\ch{\Omega_T}$, where $\Omega_T$ is the infinite triangular lattice. 

\begin{theorem}We have $\ch{\Omega_T} = 12$.
\end{theorem}

\begin{proof}
We will again find lower and upper bounds using well-spaced rows and star-elimination. It can be visually useful to restructure $\Omega_T$ such that $\Omega_S$ is clearly a subgraph of $\Omega_T$, which it is (analogously to the manner in which we represented $\Omega_H$ as a subgraph of $\Omega_S$). A simple way of doing this is to construct $\Omega_S$ and connect the top-left and bottom-right vertices of each $C_4$ by an edge, as shown in Figure \ref{square_subgraph_of_triangle}.

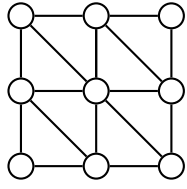
\begin{figure}
    \centering
    \begin{tikzpicture}[node distance = {10mm}, thick, main/.style = {draw, circle}]
    \node[main] (1) {};
    \node[main] (2) [right of=1] {};
    \node[main] (3) [right of=2] {};
    \node[main] (4) [below of=1] {};
    \node[main] (5) [below of=2] {};
    \node[main] (6) [below of=3] {};
    \node[main] (7) [below of=4] {};
    \node[main] (8) [below of=5] {};
    \node[main] (9) [below of=6] {};
    \draw (1) -- (2);
    \draw (2) -- (3);
    \draw (4) -- (5);
    \draw (5) -- (6);
    \draw (7) -- (8);
    \draw (8) -- (9);
    \draw (2) -- (5);
    \draw (1) -- (4);
    \draw (2) -- (5);
    \draw (3) -- (6);
    \draw (4) -- (7);
    \draw (5) -- (8);
    \draw (6) -- (9);
    \draw (1) -- (5);
    \draw (2) -- (6);
    \draw (4) -- (8);
    \draw (5) -- (9);
    
    \end{tikzpicture}
    \caption{Representation of $\Omega_T$ as a graph containing $\Omega_S$ as a subgraph.}
    \label{square_subgraph_of_triangle}
\end{figure}
Because each vertex of $\Omega_T$ is the center of a $K_{1,6}$ ($\Omega_T$ is 6-regular) and $\ch{K_{1,6}}=7$, $\ch{\Omega_T} \geq 7$, we will need to use seven distinct labels. In fact, we can do better: because $\Omega_S$ is a subgraph of $\Omega_T$, $\ch{\Omega_T} \geq \ch{\Omega_S}=8$. 

Star-elimination quickly increases the lower bound on $\ch{\Omega_T}$ to 11. We will now use star-elimination to increase the lower bound to 12. \\
Assume $\ch{\Omega_T}=11$. We will need to show that there are only six possible labels for the seven vertices in each $K_{1,6}$ subgraph of $\Omega_T$. We do this by eliminating six of the eleven possible labels for the neighbors of the central vertex of an arbitrary $K_{1,6}$ in a hypothetical 11-TDL of $\Omega_T$. 
\begin{enumerate}
    \item Assume 5 appears. Then we eliminate 5 and 10; we also eliminate one of 1 and 9, 2 and 8, 3 and 7, and 4 and 6, leaving 11 as well as one from each of these pairs, leaving only five possible labels for the six neighbors of 5; therefore, 5 does not appear.
    \item Assume 6 appears. Eliminate 3, 5, and 6. Also eliminate one of 1 and 11, 2 and 10, and 4 and 8. This leaves 7 and 9, plus one of each of the three pairs just checked, leaving only five possible neighbor labels; so 6 does not appear.
    \item Assume 4 appears. Eliminate 2, 3, 4, 5, 6, and 8. Also eliminate one of 1 and 7. This leaves 9, 10, and 11, plus one of 1 and 7; so 4 does not appear.
    \item Assume 9 appears. Eliminate 4, 5, and 6, as well as one of 7 and 11 and 8 and 10. This leaves 1, 2, and 3, plus two of 7, 8, 10, and 11; so 9 does not appear.
    \item Assume 2 appears. Eliminate 1, 2, 4, 5, 6, and 9. This leaves 3, 7, 8, 10, and 11; so 2 does not appear.
    \item Assume 7 appears. Eliminate 2, 4, 5, 6, and 9. Also eliminate one of 3 and 11. This leaves 1, 8, 9, 10, and one of 3 and 11; so 7 does not appear.
\end{enumerate}
We have now eliminated six labels, so $\ch{\Omega_T}>11$. \\

We use minimal well-spaced rows to find an upper bound on $\ch{\Omega_T}$. Because a TDL of $\Omega_T$ requires seven distinct labels, we will use the (unique) minimal well-spaced row with seven labels $\{1,3,4,9,10,12,13\}$. We can apply this WSR to $\Omega_T$ as in Figure \ref{triangle_lattice_13tdl}.

\begin{figure}
    \centering
     \begin{tikzpicture}[node distance = {10mm}, thick, main/.style = {draw, circle}]
     \node[main] (1) {1};
     \node[main] (2) [right of=1] {3};
     \node[main] (3) [right of=2] {4};
     \node[main] (4) [right of=3] {9};
     \node[main] (5) [right of=4] {10};
     \node[main] (6) [right of=5] {12};
     \node[main] (25) [right of=6] {13};
     \node[main] (7) [below of=1] {12};
     \node[main] (8) [below of=2] {13};
     \node[main] (9) [below of=3] {1};
     \node[main] (10) [below of=4] {3};
     \node[main] (11) [below of=5] {4};
     \node[main] (12) [below of=6] {9};
     \node[main] (26) [below of=25] {10};
     \node[main] (13) [below of=7] {9};
     \node[main] (14) [below of=8] {10};
     \node[main] (15) [below of=9] {12};
     \node[main] (16) [below of=10] {13};
     \node[main] (17) [below of=11] {1};
     \node[main] (18) [below of=12] {3};
     \node[main] (27) [below of=26] {4};
     \node[main] (19) [below of=13] {3};
     \node[main] (20) [below of=14] {4};
     \node[main] (21) [below of=15] {9};
     \node[main] (22) [below of=16] {10};
     \node[main] (23) [below of=17] {12};
     \node[main] (24) [below of=18] {13};
     \node[main] (28) [below of=27] {1};
     \draw (1) -- (2);
    \draw (2) -- (3);
    \draw (3) -- (4);
    \draw (4) -- (5);
    \draw (5) -- (6);
    \draw (6) -- (25);
    \draw (7) -- (8);
    \draw (8) -- (9);
    \draw (9) -- (10);
    \draw (10) -- (11);
    \draw (11) -- (12);
    \draw (12) -- (26);
    \draw (13) -- (14);
    \draw (14) -- (15);
    \draw (15) -- (16);
    \draw (16) -- (17);
    \draw (17) -- (18);
    \draw (18) -- (27);
    \draw (19) -- (20);
    \draw (20) -- (21);
    \draw (21) -- (22);
    \draw (22) -- (23);
    \draw (23) -- (24);
    \draw (24) -- (28);
    \draw (1) -- (7);
    \draw (2) -- (8);
    \draw (3) -- (9);
    \draw (4) -- (10);
    \draw (5) -- (11);
    \draw (6) -- (12);
    \draw (25) -- (26);
    \draw (7) -- (13);
    \draw (8) -- (14);
    \draw (9) -- (15);
    \draw (10) -- (16);
    \draw (11) -- (17);
    \draw (12) -- (18);
    \draw (26) -- (27);
    \draw (13) -- (19);
    \draw (14) -- (20);
    \draw (15) -- (21);
    \draw (16) -- (22);
    \draw (17) -- (23);
    \draw (18) -- (24);
    \draw (27) -- (28);
    \draw (1) -- (8);
\draw (2) -- (9);
\draw (3) -- (10);
\draw (4) -- (11);
\draw (5) -- (12);
\draw (6) -- (26);
\draw (7) -- (14);
\draw (8) -- (15);
\draw (9) -- (16);
\draw (10) -- (17);
\draw (11) -- (18);
\draw (12) -- (27);
\draw (13) -- (20);
\draw (14) -- (21);
\draw (15) -- (22);
\draw (16) -- (23);
\draw (17) -- (24);
\draw (18) -- (28);

     \end{tikzpicture}
    \caption{Subgraph of $\Omega_T$ with 13-TDL. Each row uses the elements of the minimal 7-element WSR.}
    \label{triangle_lattice_13tdl}
\end{figure}
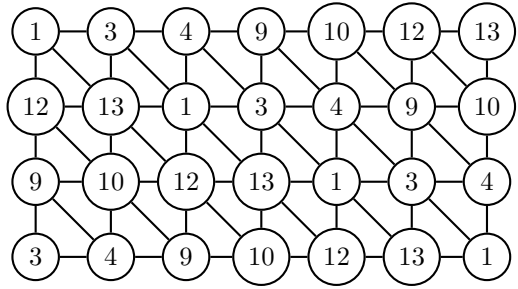

We can construct a 12-TDL of $\Omega_T$ using the labels $\{1,2,3,4,5,7,9,10,11,12\}$; see Figure \ref{triangle_lattice_12tdl}.

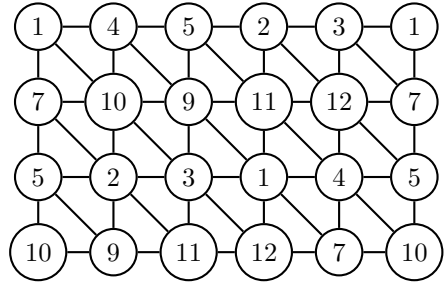
\begin{figure}
    \centering
     \begin{tikzpicture}[node distance = {10mm}, thick, main/.style = {draw, circle}]
     \node[main] (1) {1};
     \node[main] (2) [right of=1] {4};
     \node[main] (3) [right of=2] {5};
     \node[main] (4) [right of=3] {2};
     \node[main] (5) [right of=4] {3};
     \node[main] (6) [right of=5] {1};
     \node[main] (7) [below of=1] {7};
     \node[main] (8) [below of=2] {10};
     \node[main] (9) [below of=3] {9};
     \node[main] (10) [below of=4] {11};
     \node[main] (11) [below of=5] {12};
     \node[main] (12) [below of=6] {7};
     \node[main] (13) [below of=7] {5};
     \node[main] (14) [below of=8] {2};
     \node[main] (15) [below of=9] {3};
     \node[main] (16) [below of=10] {1};
     \node[main] (17) [below of=11] {4};
     \node[main] (18) [below of=12] {5};
     \node[main] (19) [below of=13] {10};
     \node[main] (20) [below of=14] {9};
     \node[main] (21) [below of=15] {11};
     \node[main] (22) [below of=16] {12};
     \node[main] (23) [below of=17] {7};
     \node[main] (24) [below of=18] {10};
     \draw (1) -- (2);
    \draw (2) -- (3);
    \draw (3) -- (4);
    \draw (4) -- (5);
    \draw (5) -- (6);
    \draw (7) -- (8);
    \draw (8) -- (9);
    \draw (9) -- (10);
    \draw (10) -- (11);
    \draw (11) -- (12);
    \draw (13) -- (14);
    \draw (14) -- (15);
    \draw (15) -- (16);
    \draw (16) -- (17);
    \draw (17) -- (18);
    \draw (19) -- (20);
    \draw (20) -- (21);
    \draw (21) -- (22);
    \draw (22) -- (23);
    \draw (23) -- (24);
    \draw (1) -- (7);
    \draw (2) -- (8);
    \draw (3) -- (9);
    \draw (4) -- (10);
    \draw (5) -- (11);
    \draw (6) -- (12);
    \draw (7) -- (13);
    \draw (8) -- (14);
    \draw (9) -- (15);
    \draw (10) -- (16);
    \draw (11) -- (17);
    \draw (12) -- (18);
    \draw (13) -- (19);
    \draw (14) -- (20);
    \draw (15) -- (21);
    \draw (16) -- (22);
    \draw (17) -- (23);
    \draw (18) -- (24);
    \draw (1) -- (8);
\draw (2) -- (9);
\draw (3) -- (10);
\draw (4) -- (11);
\draw (5) -- (12);
\draw (7) -- (14);
\draw (8) -- (15);
\draw (9) -- (16);
\draw (10) -- (17);
\draw (11) -- (18);
\draw (13) -- (20);
\draw (14) -- (21);
\draw (15) -- (22);
\draw (16) -- (23);
\draw (17) -- (24);

     \end{tikzpicture}
    \caption{Subgraph of $\Omega_T$ with 12-TDL. Alternate rows of vertices use the same five labels.}
    \label{triangle_lattice_12tdl}
\end{figure}
\end{proof}

We also have a result for the graph obtained from the cubic lattice. Set $\Omega_{Q_3}$ to be the cubic lattice graph. 

\begin{theorem} We have $12 \leq \ch{\Omega_{Q_3}} \leq 13$.
\end{theorem}
Here the upper bound is obtained from a well-spaced row, and the lower bound is obtained from star-elimination.


We also find the $\ch{G}$ for the infinite complete binary tree graph, defined as the graph starting with a single vertex and where every vertex has two children, repeating infinitely. 

\begin{theorem} Let $B$ be the graph of the infinite complete binary tree. We have $\ch{B} = 7$.
\end{theorem}

\begin{proof} Since this graph is 3-regular except at the base of the tree, we can use the same method as with the infinite hexagonal lattice to obtain that $\ch{B} \geq 7$.

To show that $\ch{G} \leq 7$, we give a specific labeling for this graph. We will supply a pair of numbers for each label used in the labeling. We label the base of the tree with 1. 
We then label each pair of vertices in the tree use the following set of rules. Each pair after the arrow indicates the two corresponding labels for the next pair of vertices. \\ \\
$1 \xrightarrow{} {3, 5}$\\
$2 \xrightarrow{} {5, 7}$\\
$3 \xrightarrow{} {2, 7}$\\
$4 \xrightarrow{} {6, 7}$\\
$5 \xrightarrow{} {4, 7}$\\
$6 \xrightarrow{} {1, 2}$\\
$7 \xrightarrow{} {1, 6}$\\

It is straightforward to check that the resulting labeling of the tree is a TDL. 

\end{proof}

One question that is implicitly raised by the results in this section is when an infinite graph has a finite total difference labeling number. Clearly if a graph  has arbitrarily high degree vertices then it cannot have a finite total difference labeling number (for that matter it cannot even have a finite coloring number).  This is essentially the only circumstance where the total difference labeling fails to exist. In particular we have:

\begin{theorem}\label{Infinite labelings exist} Let $G$ be a countable, infinite graph, where $\Delta(G)$ is finite. Then $\ch{G}$ is defined. Moreover, let $M$ be the largest element of a minimal well-spaced row with $\Delta(G)^2 +1 $ elements. Then $\ch{G} \leq M$.
\end{theorem}
\begin{proof} Assume $G$ is a countable, infinite graph with $\Delta(G)$ finite. Let $S$ be a well-spaced row with greatest element $M$ and let the vertices of $G$ be $v_1$, $v_2$, $v_3  \cdots$. 

We set $N_i$ to be the set of all vertices that are distance one or distance two away from $v_i$. Note that there are at most $\Delta(G)^2$ elements in any $N_i$.

Then we label the vertices inductively, assigning to each $v_i$ the smallest element of $S$ that has not yet been assigned to any label in $N_i$. Since $N_i$ itself has at most $\Delta(G)^2$ elements, and $S$ has  $\Delta(G)^2+1$ elements, we can always find such a label. The labeling that results is a total difference labeling. Since the labeling uses a well-spaced row, we just need to check that there are no sandwiches and no duplicate adjacent vertices, but both are ruled out since no vertex $v_j$ is ever labeled the same as any other vertex in $N_j$.
\end{proof}

The labeling given in Theorem \ref{Infinite labelings exist} has some drawbacks. It is frequently much less efficient than the ideal labeling for an infinite graph. For example, this labeling scheme would tell you that $\ch{\Omega_S} \leq 53$, since the most efficient well-spaced row with $\Delta(\Omega_S)^2+1=4^2+1=17$ elements has largest element 53. Second, the specific labeling you get from this is not canonical but depends sensitively on the order the vertices are listed.\\

Note also that although we have only stated Theorem \ref{Infinite labelings exist} for countable graphs, this phrasing is essentially a matter of convenience to make the proof straightforward and avoid any issues involving the Axiom of Choice. The theorem is also valid for larger cardinality graphs as long as one assumes Choice.

 \section{Cloning and hypercubes}

 Let $Q_d$ be the graph obtained from the $d$-dimensional hypercube. That is, $Q_0$ is a single vertex, $Q_1$ is the graph of two connected vertices, $Q_2$ is the square, and so on. \\
 
 Straightforward computation establishes that 
$\ch{Q_0}=1$, $\ch{Q_1}=3$, $\ch{Q_2} = 5$, $\ch{Q_3} = 7$, and $\ch{Q_4} = 9$. At this point, one might wish to guess that $\ch{Q_5} = 11$. Alas, this is not the case. In fact, $\ch{Q_5}=10$. We cannot give a complete description of $\ch{Q_d}$ but will prove an estimate using a concept we call cloning.  \\
 
  Given a graph G with vertices $g_1, g_2, g_3 \cdots $ (with possibly infinitely many vertices), we will define the {\emph {clone of G}} to 
be the Cartesian product of $G$ with $K_2$. This is the graph made by $x_1, x_2, x_3 \cdots $ and $y_1, y_2, y_3 \cdots $ with edges given by the following:

\begin{enumerate}
    \item $x_i$ is connected to $x_j$ if and only if $g_i$ is connected to $g_j$.
     \item $y_i$ is connected to $y_j$ if and only if $g_i$ is connected to $g_j$.
     \item $x_i$ is connected to $y_j$ if and only if $i=j$.
\end{enumerate}

In other words, to clone a graph $G$, make a copy of $G$ and connect each vertex of $G$ to its corresponding copy. Notice that the series of hypercubes is obtained by cloning the trivial graph. Given a graph $G$, we will write $\clone{G}$ to be its clone. Thus, for example, if $G$ is the square graph $Q_2$, then $\clone{G}$ is the 3-D cube graph $Q_3$, and $\clone{\clone{G}}$ is the 4-D hypercube graph $Q_4$, and so on. \\

Estimating $\ch{G}$ for hypercubes will rely on the following lemma:

\begin{lemma}\label{cloning lemma} If $G$ is a graph,  then $\ch{\clone{G}} \leq  2\ch{G}+1$.
\end{lemma}
\begin{proof} Let $f(G)$ be a TDL function of $G$. Let $H= \clone{G}$. We label $H$  with labeling
function $h$ defined as follows: For each vertex $g_i$ in $G$,  $h(x_i) = f(g_i)$ and  $h(y_i) = f(g_i) + X_td + 1$ (where, recall, the $x_i$ and $y_j$ are corresponding copies of vertices).
We claim that this is a TDL with largest label $2\ch{G}+1$. It is immediately apparent that $h(H)$ has maximum label $\ch{G} + \ch{G} +1 =2\ch{G}+1$. Thus,  we just need to check that there are no doubles, sandwiches or staircases. \\

There are no doubles among the set of  $x_i$ because they are directly labeled from our total difference labeling from G. There are no doubles among the $y_i$ because the smallest value of any $h(y_i)$ is greater than half the largest value of the $y_i$ labels. There is no double going from an $x_i$ to a $y_i$ because all the labels for the $y_i$ are greater than twice the largest $x_i$ value.\\

There are no sandwiches among the $x_i$ because they again have the same labeling as in $G$. There are no sandwiches among the $y_i$ because each of the $y_i$ are all the same values as the $x_i$ but increased by a constant. There are three possible ways there could be a sandwich with a combination of the $x_i$ and the $y_i$. First, there could be a sandwich of the form $x_a$, $y_b$, $x_c$, but this cannot happen because there is no set of connected vertices of that form. Second, there could be a sandwich of the form $y_a$, $x_b$, $y_c$, but again there are no connected vertices of that form. The third possible form of a sandwich is $x_a$, $z$, $y_b$ where $z$ may be either an $x$ or a $y$ vertex. But such a sandwich would require that $h(x_a) = h(y_b)$, which is never true.  \\

There are no staircases among the $x_i$ because the $x_i$ inherited their labels from the labeling of $G$. There are no staircases among the $y_i$ because their labels are all a constant up from the $x_i$ labels. There are no staircases involving both $x_i$ and $y_i$ because the difference between the largest $x_i$ and smallest $y_i$ is larger than the smallest difference between any  $x_i$ (which is also the smallest difference between any $y_i$). 
\end{proof}

We can apply Lemma \ref{cloning lemma} inductively to the hypercubes to get the following result: 

\begin{lemma}\label{Hypercube bound} For all $d$ we have $\ch{Q_d} \leq 2^{d+1}-1$. 
\end{lemma}

It is pretty clear that Lemma \ref{Hypercube bound} gives what is often a weak bound. For example, we know that $\chi_{td}$ of the square lattice is 8. This would tell us that clone of the square lattice has $\chi_{td}$ at most 17. But the clone of the square lattice is a subgraph of the cubic lattice where we know $\chi_{td}$ is at most 13. In this case, Lemma \ref{cloning lemma} is giving a significant overestimate of the actual value of $\chi_td$.\\

When is it that $\ch{cl(G)}=2\ch{G}+1$? Are there infinitely many graphs with this property?  The only graph we are aware of where this bound is exactly equal is when $G$ is a lone vertex.  For certain families of graphs we can prove explicitly that this bound is weak. In the case of a path graph, the clone is just a cycle graph, and so that this bound is not best possible follows immediately from Proposition \ref{RZ cycle result} and Proposition \ref{RZ star result}. The next two results show that this bound is not best possible for complete graphs and star graphs.  \\

\begin{proposition} Let $n \geq 3$. Then $\ch{cl(K_n)} \leq  2\ch{K_n}$.
\end{proposition}
\begin{proof} Assume we have the graph $K_n$ with its most efficient total difference labeling. This labeling for $K_n$ is then the minimal well-spaced row with $n$ elements. We now consider the graph $cl(K_n)$ with one copy of $K_n$ labeled $v_1, v_2 \cdots v_n$ and the other  labeled $u_1, u_2 \cdots u_n$, and with $v_i$ connected to $u_i$ for $i$ satisfying $1 \leq i \leq n$.
We assign to each of the $v_i$ one of the labels from our well-spaced row from $K_n$, in increasing order, so $f(v_1)$ is smallest label and $f(v_n)$ has our largest label. We note that since this is a well-spaced row we must either be missing $1$ as a label or must be missing $2$ as a label in our well-spaced row. \\ 

Assume that $2$ is missing in our well-spaced row. Then for $1 \leq i \leq n-1$ we set $f(u_i) = f(v_i) + \ch{K_n}$, and set $f(u_n) =2$. We cannot have a double or a triple among the $v_i$ because the $v_i$ form a well-spaced row.  Since $f(v_n) >4$ we cannot have a double  between $v_n$ and $u_n$, and we cannot have a double or a staircase between the other $u_i$ and $v_i$ by the same logic as we had with our basic cloning argument in the proof of Lemma \ref{cloning lemma}. Finally, we cannot have a staircase involving $u_n$ since the only possible staircase involving $2$ is $1-2-3$ which cannot occur here.\\

The case where $1$ is the missing label is similar. 
\end{proof}

One obvious question is  if one has a random graph (in the Erd\H{o}s-Renyi sense), is it true that this Lemma is, with probability 1, very weak?  \\

\begin{conjecture} For any $\epsilon>0$, given the  Erd\H{o}s-Renyi random graph model, with probability 1, for a random graph $G$ and its clone $H$, $\ch{H} \leq (1+\epsilon)\ch{G}$.

\end{conjecture}

Recall that the clone of a graph is its Cartesian product with $K_2$ and that the Cartesian product of two graphs $G_1,$ $G_2$ is the graph $H=G_1 \square G_2$ satisfying
\begin{enumerate}
    \item The set of vertices of $H$ is the Cartesian product of $V(G_1)$ and $V(G_2)$
    \item A vertex $(g_1,g_2) \in V(H)$ is adjacent to another vertex $(g_1', g_2') \in V(H)$ if and only if either $g_1=g_1'$ and $g_2$ is adjacent to $g_2'$ in $G_2$, or $g_2=g_2'$ and $g_1$ is adjacent to $g_1'$ in $G_1$.
\end{enumerate}

Using similar logic as Lemma \ref{cloning lemma} one can prove:

\begin{theorem} Let $H=  K_m \square G$ for some $m$. Then $\ch{H} \leq m\ch{G} + \frac{m(m-1)}{2}.$
\end{theorem}

It seems natural to ask the following:

\begin{question} Is there a function $f(x,y)$ such that for any graphs $G_1$ and $G_2$ we have $\ch{G_1 \square G_2} \leq f(\ch{G_1}, \ch{G_2})$?
\end{question}

Another natural question is to ask whether it is always true for any graph $G$ that $\ch{cl(G)} > \ch{G}.$ This inequality is in fact false for the Petersen graph. If $G$ is the Petersen graph, then $\ch{cl(G)} = \ch{G} = 10$. This raises the following question:

\begin{question} Is there a graph $G$ such that for any non-empty graph $H$,  we have $\ch{H \square G} > \ch{G}$?
\end{question}

Other notions of graph products exist, including the tensor product and strong product, and it may be natural to ask similar questions about how total difference labeling interacts with those products. 

\section{Saturated graphs}

 We say that a total difference labeling of a graph $G$ with order $n$ is a {\it saturated labeling} if the TDL's labels are exactly $\{1,2,3 \cdots \ch{G}\}$, and  $\ch{G}=n$. We define a {\it saturable graph} as a graph that has at least one saturated labeling. Examples of saturable graphs are $C_5$ and the Petersen graph. \\

The graph $P_4$ is an example of a special case of a saturable graph. Note that $\ch{P_4}=4$, and that it has two 4-TDLs, one using the labels $\{1,3,4\}$ and one using $\{1,2,3,4\}$ (see Figures \ref{nonsaturated_P4} and \ref{saturated_P4_1}).\\
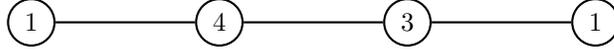
\begin{figure}
     \centering
     \begin{tikzpicture}[node distance = {25mm}, thick, main/.style = {draw, circle}]
     \node[main] (1) {1};
     \node[main] (2) [right of=1] {4};
     \node[main] (3) [right of=2] {3};
     \node[main] (4) [right of=3] {1};
     \draw (1) -- (2);
     \draw (2) -- (3);
     \draw (3) -- (4);
     \end{tikzpicture}
     \caption{A non-saturated labeling of $P_4$.}
     \label{nonsaturated_P4}
 \end{figure}
 \begin{figure}
     \centering
     \begin{tikzpicture}[node distance = {25mm}, thick, main/.style = {draw, circle}]
     \node[main] (1) {4};
     \node[main] (2) [right of=1] {1};
     \node[main] (3) [right of=2] {3};
     \node[main] (4) [right of=3] {2};
     \draw (1) -- (2);
     \draw (2) -- (3);
     \draw (3) -- (4);
     \end{tikzpicture}
     \caption{A saturated labeling of $P_4$.}
     \label{saturated_P4_1}
 \end{figure}
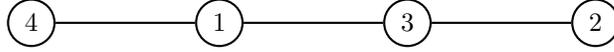
 
Note that the first of these labelings does not use the label 2, and is therefore not a saturated labeling. However, the second labeling is a saturated labeling.\\

For $C_5$, the saturated labeling is the only minimal labeling for the graph. Motivated by this difference in behavior of $P_4$ and $C_5$,  we call a graph $G$ {\it supersaturable} if every one of its $\ch{G}$-total difference labelings is a saturated labeling.\\

\begin{theorem} Let $G$ be a saturable graph with diameter at most 2. Then $G$ is supersaturable.
\end{theorem}
\begin{proof}
Assume $G$ is a saturable graph with diameter at most 2. If a graph has diameter at most 2, then each vertex is at most distance 2 from every other vertex. Therefore, using a label twice in the graph would always cause a sandwich. Since all vertices must be labeled with different labels, and the graph has a proper saturated labeling, all of its minimal labelings must be saturable.
\end{proof}

One might imagine that in a minimal total distancing labeling of a graph, any pair of vertices with distance greater than 2 can always be labeled with the same label. Since they are sufficiently far apart, the two will never be a part of the same labeling obstruction (a double or triple) and therefore labeling them identically should not cause any issues in the final labeling. However, this is not the case, as there are supersaturable graphs with diameter greater than 2. An example is the modified ``triforce graph'', shown in Figure \ref{triforce1}, whose four 6-TDLs, all saturated, are shown in Figure \ref{triforce_labelings}. \\

\begin{figure}
    \centering
    \begin{tikzpicture}[node distance = {15mm}, thick, main/.style = {draw, circle}] 
    \node[main] (1) {};
    \node[main] (2) [below left of=1] {};
    \node[main] (3) [below right of=1] {};
    \node[main] (4) [below left of=2] {};
    \node[main] (5) [below left of=3] {};
    \node[main] (6) [below right of=3] {};
    \draw (1) -- (2);
    \draw (1) -- (3);
    \draw (2) -- (3);
    \draw (2) -- (5);
    \draw (3) -- (5);
    \draw (3) -- (6);
    \draw (4) -- (5);
    \draw (5) -- (6);
    \end{tikzpicture}
    \caption{A triforce graph missing an edge.}
    \label{triforce1}
\end{figure}
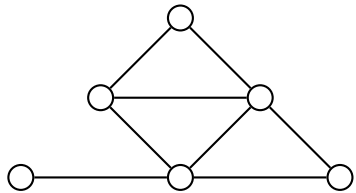

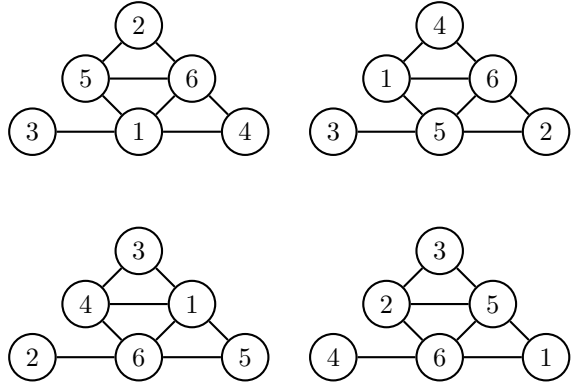
\begin{figure}
    \centering
    \begin{tikzpicture}[node distance = {10mm}, thick, main/.style = {draw, circle}] 
    \node[main] (1) {3};
    \node[main] (2) [below left of=1] {4};
    \node[main] (3) [below right of=1] {1};
    \node[main] (4) [below left of=2] {2};
    \node[main] (5) [below left of=3] {6};
    \node[main] (6) [below right of=3] {5};
    \draw (1) -- (2);
    \draw (1) -- (3);
    \draw (2) -- (3);
    \draw (2) -- (5);
    \draw (3) -- (5);
    \draw (3) -- (6);
    \draw (4) -- (5);
    \draw (5) -- (6);
    \begin{scope}[shift={(4,0)}]
        \node[main] (1) {3};
        \node[main] (2) [below left of=1] {2};
        \node[main] (3) [below right of=1] {5};
        \node[main] (4) [below left of=2] {4};
        \node[main] (5) [below left of=3] {6};
        \node[main] (6) [below right of=3] {1};
        \draw (1) -- (2);
        \draw (1) -- (3);
        \draw (2) -- (3);
        \draw (2) -- (5);
        \draw (3) -- (5);
        \draw (3) -- (6);
        \draw (4) -- (5);
        \draw (5) -- (6);
        \begin{scope}[shift={(-4,3)}]
            \node[main] (1) {2};
            \node[main] (2) [below left of=1] {5};
            \node[main] (3) [below right of=1] {6};
            \node[main] (4) [below left of=2] {3};
            \node[main] (5) [below left of=3] {1};
            \node[main] (6) [below right of=3] {4};
            \draw (1) -- (2);
            \draw (1) -- (3);
            \draw (2) -- (3);
            \draw (2) -- (5);
            \draw (3) -- (5);
            \draw (3) -- (6);
            \draw (4) -- (5);
            \draw (5) -- (6);
            \begin{scope}[shift={(4,0)}]
                \node[main] (1) {4};
                \node[main] (2) [below left of=1] {1};
                \node[main] (3) [below right of=1] {6};
                \node[main] (4) [below left of=2] {3};
                \node[main] (5) [below left of=3] {5};
                \node[main] (6) [below right of=3] {2};
                \draw (1) -- (2);
                \draw (1) -- (3);
                \draw (2) -- (3);
                \draw (2) -- (5);
                \draw (3) -- (5);
                \draw (3) -- (6);
                \draw (4) -- (5);
                \draw (5) -- (6);
            \end{scope}
        \end{scope}
    \end{scope}
    \end{tikzpicture}
    \caption{Saturated labelings of the triforce graph.}
    \label{triforce_labelings}
\end{figure}

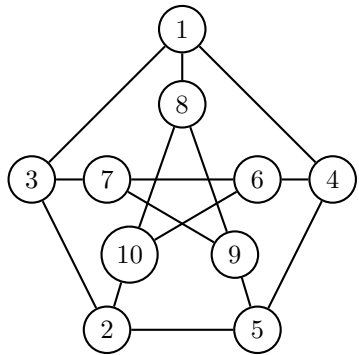
\begin{figure}
    \centering
    \begin{tikzpicture}[thick, main/.style = {draw, circle}]
    \node[main] (1) at (2,5) {1};
    \node[main] (2) at (1,1) {2};
    \node[main] (3) at (0,3) {3};
    \node[main] (4) at (4,3) {4};
    \node[main] (5) at (3,1) {5};
    \node[main] (8) at (2,4) {8};
    \node[main] (9) at (2.7,2) {9};
    \node[main] (10) at (1.3,2) {10};
    \node[main] (6) at (3,3) {6};
    \node[main] (7) at (1,3) {7};
    \draw (1) -- (3);
    \draw (1) -- (4);
    \draw (1) -- (8);
    \draw (2) -- (3);
    \draw (2) -- (5);
    \draw (2) -- (10);
    \draw (3) -- (7);
    \draw (4) -- (5);
    \draw (4) -- (6);
    \draw (5) -- (9);
    \draw (6) -- (7);
    \draw (6) -- (10);
    \draw (7) -- (9);
    \draw (8) -- (9);
    \draw (8) -- (10);
    \end{tikzpicture}
    \caption{A total difference labeling of the Petersen graph.}
    \label{petersen_with_labeling}
\end{figure}

We used a computer algorithm to test all 112 unique order-6 graphs and found 32 such graphs $G$ where $\ch{G}=6$. These 32 graphs graphs were then individually tested for saturability, and a saturated labeling was found for each of the graphs; therefore, all order-6 graphs with $\chi_{td}=6$ are saturable.\\

It is easier to find  by hand all saturable graphs with order less than six: There exist only seven of them. It is left as an exercise to the reader to find these graphs.\\

All of the saturable graphs with order less than six have diameter at most 2 and are therefore all supersaturable. However, even though not all saturable order-6 graphs have diameter at most 2, some are still supersaturable.\\

One might wonder if for any graph $G$ with order $n$ with $\ch{G}=n$, $G$ is saturable. While this is true for all graphs with order at most $6$, we can produce counterexamples of larger order. For example, consider the graph $I$ given in Figure \ref{graphI}, which is made from two $K_4$'s, with one vertex in each connected to the other by an edge. \\

\begin{figure}
    \centering
    \begin{tikzpicture}[node distance = {25mm}, thick, main/.style = {draw, circle}]
    \node[main] (1) {};
    \node[main] (2) [right of=1] {};
    \node[main] (3) [below of=1] {};
    \node[main] (4) [right of=3] {};
    \node[main] (5) [right of=2] {};
    \node[main] (6) [right of=5] {};
    \node[main] (7) [below of=5] {};
    \node[main] (8) [right of=7] {};
    \draw (1) -- (2);
    \draw (1) -- (3);
    \draw (1) -- (4);
    \draw (2) -- (3);
    \draw (2) -- (4);
    \draw (3) -- (4);
    \draw (5) -- (6);
    \draw (5) -- (7);
    \draw (5) -- (8);
    \draw (6) -- (7);
    \draw (6) -- (8);
    \draw (7) -- (8);
    \draw (4) -- (7);
    \end{tikzpicture}
    \caption{Graph $I$.}
    \label{graphI}
\end{figure}

\begin{proposition}
Where $I$ is the graph in Figure \ref{graphI}, $\ch{I}=8$, $I$ is connected, $I$ has order $8=\ch{I}$, and $I$ is not saturable. 
\end{proposition}
\begin{proof}
 Note that $\ch{I} \geq 8$ as $\ch{K_4}=8$. The labeling in Figure \ref{graphIwithTDL} shows that $\ch{I} \leq 8$.\\
 
 Note further that $I$ has order 8. However, $I$ is not a saturable graph since any TDL of a copy of $K_4$ which uses labels at most 8 must contain an 8 as a label. Thus, $I$'s most efficient labeling contains two 8s and is thus not saturable.
  \end{proof}
  
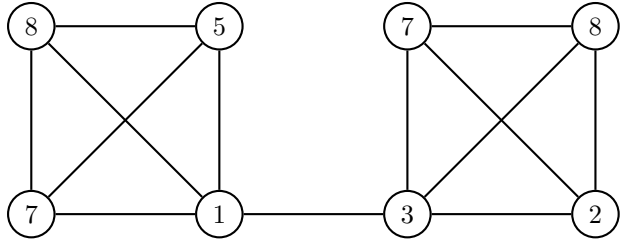
\begin{figure}
    \centering
    \begin{tikzpicture}[node distance = {25mm}, thick, main/.style = {draw, circle}]
    \node[main] (1) {8};
    \node[main] (2) [right of=1] {5};
    \node[main] (3) [below of=1] {7};
    \node[main] (4) [right of=3] {1};
    \node[main] (5) [right of=2] {7};
    \node[main] (6) [right of=5] {8};
    \node[main] (7) [below of=5] {3};
    \node[main] (8) [right of=7] {2};
    \draw (1) -- (2);
    \draw (1) -- (3);
    \draw (1) -- (4);
    \draw (2) -- (3);
    \draw (2) -- (4);
    \draw (3) -- (4);
    \draw (5) -- (6);
    \draw (5) -- (7);
    \draw (5) -- (8);
    \draw (6) -- (7);
    \draw (6) -- (8);
    \draw (7) -- (8);
    \draw (4) -- (7);
    \end{tikzpicture}
    \caption{Graph $I$ with non-saturated labeling.}
    \label{graphIwithTDL}
\end{figure}

$I$ is one graph in a general class of counterexamples. When there are at least two minimal well-spaced rows of size $n$ (\textit{i.e.}, $D(n) \geq 2)$) we may often construct a counterexample this way by connecting two copies of $K_n$ as we did with $I$. \\

\end{document}